\documentclass[a4paper,11pt]{article}
\usepackage{amssymb}
\usepackage{amsmath}
\usepackage{amsthm}

\newtheorem{prethmm}{{\bf Theorem}}

\newtheorem{prethm}{{\bf Theorem}}[section]

\newenvironment{thm}{\begin{prethm}{\hspace{-0.5
em}{\bf}}}{\end{prethm}}
\newtheorem{precor}[prethm]{{\bf Corollary}}

\newenvironment{cor}{\begin{precor}{\hspace{-0.5
em}{\bf}}}{\end{precor}}
\newtheorem{preque}[prethm]{{\bf Question}}

\newtheorem{prelemma}[prethm]{{\bf Lemma}}

\newtheorem{preex}[prethm]{{\bf Example}}

\newtheorem{prepro}[prethm]{{\bf Proposition}}

\newenvironment{pro}{\begin{prepro}{\hspace{-0.5
em}{\bf}}}{\end{prepro}}
\newtheorem{preconj}[prethm]{{\bf Conjecture}}

\newtheorem{predeff}[prethm]{{\bf Definition}}

\newtheorem{preobv}[prethm]{{\bf Observation}}

\newtheorem{preremark}[prethm]{{\bf Remark}}

\newtheorem{preclaim}{{\rm Claim}}[prethm]

\newcommand{\tl}[1]{\theta_L(#1)}
\newcommand{\ttl}[1]{\Theta_L(#1)}
\newcommand{\ttkl}[1]{\Theta_{L,K}(#1)}

\newcommand{\rk}[1]{{\rm rank}(#1)}
\newcommand{\mr}[1]{{\rm mr}(#1)}
\newcommand{\bmr}[1]{{\rm bmr}(#1)}
\newcommand{\mrp}[1]{{\rm mr_p}(#1)}
\newcommand{\bmrp}[1]{{\rm bmr_p}(#1)}
\newcommand{\mrf}[2]{{\rm mr_{#1}}(#2)}
\newcommand{\bmrf}[2]{{\rm bmr_{#1}}(#2)}

\begin{document}

\title{Some lower bounds for the $L$-intersection number of graphs}

\author{Zeinab Maleki and Behnaz Omoomi\\[2mm]
{\small Department of Mathematical Sciences}\\
{\small Isfahan University of Technology}\\
{\small 84156-83111,  Isfahan, Iran}}

\date{}

\maketitle \vspace*{-3mm}

\begin{abstract}
For a set of non-negative integers~$L$, the $L$-intersection number of a graph  is the smallest number~$l$ for which there is an assignment on the vertices to subsets $A_v \subseteq \{1,\dots, l\}$, such that every two vertices $u,v$ are adjacent if and only if $|A_u \cap A_v|\in L$.
The bipartite $L$-intersection number is defined similarly when the conditions are considered only for the vertices in different parts.
In this paper, some lower bounds for the (bipartite) $L$-intersection number of a graph for various types $L$ in terms of the minimum rank of graph are obtained.
\end{abstract}


\maketitle

\section{Introduction}

A {\it graph representation} is an assignment on the vertices of graph to a family of objects satisfying certain conditions and a rule which determines from the objects whether or not two vertices are adjacent. In the literature, different types of graph representations such as the set intersection representation~\cite{small_dim,Jukna_theta} and the vector representation~\cite{Lovasz_Shanon,Lovasz_Survey,general_vector_def} are studied.

The {\it set intersection representation} is one of the basic graph representations, which is an assignment of sets to the vertices such that two vertices are adjacent if and only if the size of the intersection of their corresponding sets satisfies the certain rule.
Precisely, let $G$ be a finite simple graph with vertex set $V$ and $L$ be a subset of non-negative integers. An {\it $L$-intersection representation} of $G$, assign to every vertex $v \in V$ a finite set $A_v$, such that two vertices $u$ and $v$ are adjacent if and only if $|A_u \cap A_v| \in L$.
The question we are interest in is the minimum size of the universe of the sets, i.e $|\cup_{v\in V}A_v|$. This parameter is denoted by $\Theta_L(G)$ and called the {\it $L$-intersection number} of $G$~\cite{small_dim}.

For bipartite graph $G$ with a fixed vertex partition $V = V_1 \cup V_2$, the definition can be modified by relaxing the condition inside the partition sets (since for vertices inside a partite set, we know they are not adjacent). Indeed, a {\it bipartite $L$-intersection representation} of graph $G$, for a given set $L \subseteq \{0,1,2,\dots\}$, assign to every vertex $v \in V$ a finite set $A_v$, such that two vertices $u,v$ from different partite sets are adjacent if and only if $|A_u \cap A_v| \in L$. The relaxed measure of the $L$-intersection number is denoted by lower case theta, $\theta_L(G)$~\cite{Jukna_theta}. It is clear that $\Theta_L(G) \ge \theta_L(G)$ for every bipartite graph $G$ and set $L$.

One of the important measure regarding set intersection representation is finding the most optimal representation for a graph by considering different sets $L$. Indeed, the {\it absolute dimension} of $G$ is defined as $\Theta(G) = \min_L \Theta_L(G)$ over all sets $L$ of non-negative integers (similarly, {\it the bipartite absolute dimension} is $\theta(G) = \min_L \theta_L(G)$). Finding explicit lower bounds for absolute dimension has important consequence in the complexity theory~\cite{Jukna_theta,Pudlak_Rodl,Razborov}. Howevere, by an easy counting argument, it was shown that there exist graphs of order $n$ with absolute dimension $\Omega(n)$. With this motivation we are interested in finding lower bounds for the various $L$-intersection number of graphs.

A twin-free graph is a graph without any pair of vertices with $N(u)-\{v\} = N(v)-\{u\}$, where $N(x)$ is the set of  vertices adjacent to $x$. As a matter of fact, for every twin-free graph $G$ of order $n$, $\Theta(G) \ge \log_2 n$. This lower bound is obtained from the fact that in such a graph no pair of vertices could be assigned the same set in a representation. Although, this lower bound is obtained simply, the question of finding an explicit construction for graph $G$ such that $\theta(G) \gg \log_2 n$ or even $\Theta(G) \gg \log_2 n$, is going to be a very challenging problem~\cite{Eaton_thesis,Jukna_theta}. It is easy to see that if $H$ is a maximal twin-free induced subgraph of $G$, then $\Theta_L(H) \leq \Theta_L(G)$ and $\theta_L(H) = \theta_L(G)$, for every set $L$. Thus, every lower bound for the $L$-intersection number of $H$ is a lower bound for the $L$-intersection number of $G$. Through this paper we consider twin-free graphs with no isolated vertex.

A good summary on the known results on the $L$-intersection number is given in~\cite{Jukna_theta} (for more results in this subject see~\cite{Eaton_Unbalanced,Eaton_Gould_Rodl,Eaton_Grable,2_path}). The most studied problems in this concept are related to the threshold type $L=\{1,2,\dots\}$ which in general case is known as the edge clique covering number, denoted by $\Theta_1(G)$. Despite to the old literature of the problem, the only known general lower bound was proved as follows for the case $\L=\{0,1,\dots,k\}$.

\begin{thm}\label{lower_p} {\rm\cite{Eaton_Gould_Rodl}}
Let $L=\{0,1,\dots, k-1\}$ for some integer $k$. Then, for every graph $G$, $\Theta_L(G^c) \ge (\Theta_1(G))^{1/k}$
\end{thm}

Similarly, the bipartite $L$-intersection number for $L=\{1,2,\dots\}$ corresponds to the well-known parameter, the edge biclique covering number~\cite{Jukna_bc}.
The bipartite $L$-intersection number for various sets $L$ are studied in~\cite{Jukna_theta} and the following lower bounds are obtained.

\begin{thm}\label{review} {\rm\cite{Jukna_theta}}
Let $p$ be a prime and $R$ be a subset of residues module $p$ with $|R|=s$. If \linebreak $L = \{l : l \pmod p \in R\}$, then for every graph $G$ of order $n$ and maximum degree $\Delta$,

\begin{enumerate}
\renewcommand{\theenumii}{\roman{enumii}}
\item $\tl{G^c} \geq (n/\Delta)^{1\over s}.$
\item $\tl{G} \geq ({1\over s}n/\Delta)^{1\over p-1}.$
\end{enumerate}

\end{thm}

Note that, the type such the ones appears in Theorem~\ref{review} is called modular type.

In this paper, we concern on finding lower bounds for (bipartite) $L$-intersection number of graphs for various types $L$. For this purpose, our main tools are linear algebra techniques via inclusion matrices. So we show how these techniques are strength in order to give elegant proofs and stronger results.

The structure of the paper is as follows. First, in Section~$2$, we present some preliminaries that we need through the paper. Then, in Section~$3$, we obtain some lower bounds for $L$-intersection number for modular types and finite sets $L$. By the similar method, in Section~$4$, we find some lower bounds for the bipartite $L$-intersection number which improve the bounds in Theorem~\ref{review}. Finally, in Section~$5$, we consider the uniform intersection set representation of graphs, where all sets assigned to the vertices have the same size.

\section{Preliminaries}

In this section, we present some definitions and known results, which are necessary to prove our main theorems. We start with the definition of the rank of a graph.

Let ${\cal M}_n(\Bbb{F})$ be the set of all $n \times n$ matrices over a field $\Bbb{F}$ and ${\cal S}_n(\Bbb{F})$ be the subset of all symmetric matrices in ${\cal M}_n(\Bbb{F})$. For $A \in {\cal S}_n(\Bbb{F})$, the graph of $A$, denoted by ${\cal G}(A)$, is a graph with vertex set $\{1, \dots , n\}$ and edge set $\{ij : A_{ij} \neq 0 \mbox{ and } i\neq j\}$. Note that the diagonal of $A$ is ignored in determining $\cal{G}(A)$.

The {\it minimum rank}~\cite{mr} of a graph $G$ over a field $\Bbb{F}$ is defined to be
$$\mrf{\Bbb{F}}{G} = \min\{\rk{A} : A\in {\cal S}_n(\Bbb{F}),\ {\cal G}(A) = G\}.$$

In the case of bipartite graph, for convenience we consider the bipartite adjacency matrix. The bipartite adjacency matrix of an $n \times n$ bipartite graph $G$ with a vertex partition $V = V_1 \cup V_2$, denoted by $A_b(G)$, is a $(0,1)$-matrix whose rows correspond to the vertices of $V_1$ and its columns correspond to the vertices of $V_2$, and the $(i,j)$ entry of $A_b(G)$ is $1$ if and only if vertex $i$ is adjacent to vertex $j$. For $A \in {\cal M}_n(\Bbb{F})$, the bipartite graph ${\cal G}_b(A)$ is a graph with bipartite set $V_1$ and $V_2$ corresponding to the rows and the columns of $A$, respectively, and edges $\{ij : A_{ij} \neq 0 \}$.

The {\it bipartite minimum rank} of a bipartite graph $G$ over a field $\Bbb{F}$ is defined to be
$$\bmrf{\Bbb{F}}{G} = \min\{\rk{A} : A\in {\cal M}_n(\Bbb{F}),\ {\cal G}_b(A) = G\}.$$

It can be easily seen that, for every bipartite graph $G$, $\mrf{\Bbb{F}}{G} = 2\bmrf{\Bbb{F}}{G}.$
For convenience, when $\Bbb{F}=\Bbb{R}$, we denote $\mrf{\Bbb{F}}{G}$ and $\bmrf{\Bbb{F}}{G}$ by $\mr{G}$ and $\bmr{G}$, also for $\Bbb{F}=\Bbb{Z}_p$ we denote them by $\mrp{G}$ and $\bmrp{G}$, respectively.

The following results are well-known and straightforward.

\begin{pro}\label{known_mr}
Over any field $\Bbb{F}$,
\begin{enumerate}
\renewcommand{\theenumi}{\alph{enumi}}
\item\label{itm:a} if $G = \bigcup_{i=1}^h G_i$, then $\mrf{\Bbb{F}}{G} \le \sum_{i=1}^h \mrf{\Bbb{F}}{G_i}.$
\item\label{itm:b} if $G'$ is an induced subgraph of $G$, then $\mrf{\Bbb{F}}{G'} \leq \mrf{\Bbb{F}}{G}$.
\end{enumerate}
\end{pro}

The key tools in the proof of our main theorems is the inclusion matrices of set systems.

Let ${\cal F}$ and  ${\cal T}$ be two families of subsets of set $[l] = \{1,\dots, l\}$. The $({\cal F},{\cal T})$-{\it inclusion matrix}, denoted by $I({\cal F},{\cal T})$ is a $(0,1)$-matrix whose rows and columns are labeled by the members of ${\cal F}$ and ${\cal T}$, respectively. The $(F,T)$ entry of $I({\cal F},{\cal T})$ will be $1$ or $0$ according to whether or not $T \subseteq F$. In the case that $\cal T$ is the family of all $t$-subsets of $[l] \cup \{0\}$, we denote the matrix by $I({\cal F},t)$ and call it the {\it $t$-inclusion matrix} of ${\cal F}$. When ${\cal F}$ is the family of all $i$-subsets of $[l] \cup \{0\}$, the corresponding $t$-inclusion matrix is denoted by $I(i,t)$.
Let $A_t({\cal F},{\cal T})=I({\cal F},t) \times I({\cal T},t)^T$, we call $A_t({\cal F},{\cal T})$ the {\it $t$-intersection matrix} of ${\cal F}$ and ${\cal T}$~\cite{Babai_Frankl_book}.
Indeed, $A_t({\cal F},{\cal T})$ is a $|{\cal F|} \times |{\cal T}|$ matrix where its $(F,T)$ entry is ${|F \cap T| \choose t}$. Moreover,
$$\rk{A_t({\cal F},{\cal T})} \leq \rk{I({\cal F},t)} \leq {l \choose t}.$$

The following fact is a useful relation in working with the inclusion matrix.
\begin{pro}\label{relat} {\rm \cite{Babai_Frankl_book}}
If ${\cal F}$ is a subfamily of $k$-subsets of $[l] \cup \{0\}$, then
$$I({\cal F},i) \times I(i,t) = {k-t \choose i-t} I({\cal F},t).$$
\end{pro}

\section{Lower bounds for the $L$-intersection number}

In this section, we present some lower bounds for the $L$-intersection number of a graph~$G$  for modular types and finite sets $L$ in terms of the minimum rank of $G$.

\begin{thm}\label{mod_mr}
Let $p$ be a prime number and $R$ be a subset of residues module $p$ with $|R|=s$. If $L = \{l : l \pmod p \in R\}$, then for every graph $G$,

\noindent
{\rm(i)}
$\displaystyle \mrp{G^c}\leq \sum_{t=0}^s{\ttl{G}\choose t}.$

\noindent{\rm(ii)}
$\displaystyle \mrp{G}\leq \sum_{t=0}^{p-1}{\ttl{G}\choose t}.$
\end{thm}

\begin{proof}{\belowdisplayskip=-20pt
Assume that, $R=\{r_1,r_2,\dots,r_s\}$ and ${\cal A}=\{A_1,\dots,A_n\}$ is the family of sets assigned to the vertices of $G$ in an optimal $L$-intersection representation, i.e. $A_i\subseteq \{1,\dots, \Theta_L(G)\}$. Let $M_t=A_t({\cal A},{\cal A})$, $0 \le t \le s$, be the $t$-intersection matrix of the family $\cal A$. Remark that, $M_t$ is an $n\times n$ matrix, with ${|A_u\cap A_v|\choose t}$ in the position $(u,v)$, and
$$\rk{M_t}\leq {\Theta_L(G)\choose t}.$$\\

\noindent
{\rm(i)} First, we can choose $a_t$ in $\Bbb{Z}_p$, $0\leq t\leq s$, such that, for every non-negative integer $x$,

\begin{equation}\label{1}
\prod_{t=1}^s(x-r_t) \equiv \sum_{t=0}^s a_t{x\choose t}\quad \quad  \pmod p.
\end{equation}

Then, we define an $n \times n$ matrix $M=\sum_{t=0}^s a_tM_t$.
Thus by Relation~(\ref{1}), the $(u,v)$ entry of $M$ is equal to $\prod_{t=1}^s(|A_u\cap A_v|-r_t)$ $\pmod p$. Therefore, $M$ is a symmetric matrix such that for every $u \neq v$, its $(u,v)$ entry is a multiple of $p$ if and only if vertex $u$ is adjacent to vertex $v$. Hence, over the field $\Bbb{Z}_p$,
$$\mrp{G^c} \leq \rk{M}.$$\\

On the other hand, by the definition of $M$, the row space of $M$ is a subspace of the vector space spanned by the rows of $M_t$, $0\leq t\leq s$, and consequently,
$$\rk{M} \leq \sum_{t=0}^s{\rk{M_t}} \leq \sum_{t=0}^s {\Theta_L(G)\choose t}.$$\\

Thus, $$\mrp{G^c} \leq \sum_{t=0}^s{\ttl{G}\choose t}.$$

\noindent
{\rm(ii)} Now we choose $b_t^i$ in $\Bbb{Z}_p$, where $0 \leq t \leq p-1$ and $1 \leq i \leq s$, such that,

\begin{equation}\label{2}
1-(x-r_i)^{p-1} \equiv \sum_{t=0}^{p-1} b_t^i{x\choose t}\quad \quad \pmod p.
\end{equation}

Then, we define an $n\times n$ matrix $M= \sum_{i=1}^s\sum_{t=0}^{p-1} b_t^iM_t$. Thus by Relation~(\ref{2}), the $(u,v)$ entry of $M$ is equal to $\sum_{i=1}^s [1-(|A_u \cap A_v|-r_i)^{p-1}]$. Hence, by the Fermat's little theorem, for every two vertices $u$ and $v$, the $(u,v)$ entry of $M$ is zero in $\Bbb{Z}_p$ if and only if vertex $u$ is not adjacent to vertex $v$. Therefore, similar to above,
$$\mrp{G} \leq \rk{M} \leq \sum_{t=0}^{p-1}{\rk{M_t}} \leq \sum_{t=0}^{p-1} {\Theta_L(G)\choose t}.$$
}\end{proof}

Using the following approximation for the binomial coefficients, we obtain lower bounds for $\ttl{G^c}$ and $\ttl{G}$ in terms of the minimum rank of $G$.

It can be seen that, for every positive integers $x, s > 1$, we have
\begin{equation}\label{*}
\sum_{i=0}^s{x\choose i} \leq x^s.\tag{*}
\end{equation}

\begin{cor}\label{cor_mr}
Let $p$ be a prime number and $R$ be a subset of residues module $p$ with $|R|=s$, where $s >1$. If $L = \{l : l \pmod p \in R\}$, then for every graph $G$,

\noindent{\rm (i)} $\ttl{G^c} \geq (\mrp{G})^{1\over s}.$

\noindent{\rm (ii)} $\ttl{G} \geq (\mrp{G})^{1\over p-1}.$
\end{cor}

Note that in the proof of part~(i) in Theorem~\ref{mod_mr}, if the coefficients $a_t$ in Relations~(\ref{1}), and the matrices consider over the field $\Bbb{R}$, then with the similar argument the lower bound in terms of $\mr{G}$ for $\Theta_L(G^c)$, where $L$  is a finite set, is obtained. Hence, we have the following theorem.

\begin{thm}\label{finite_mr}
If $L$ is a finite set of size $s$, where $s>1$, then for every graph $G$, $\ttl{G^c} \geq (\mr{G})^{1\over s}$.
\end{thm}


\section{Lower bounds for the bipartite $L$-intersection number}

This section deals with the bipartite $L$-intersection number of graphs for  modular types and finite types $L$. Here, by defining appropriate inclusion matrices we obtain lower bounds for $\theta_L(G)$ in terms of the bipartite minimum rank of $G$.

\begin{thm}\label{mod_bmr}
Let $p$ be a prime and $R$ be a subset of residues module $p$ with $|R|=s$. If $L = \{l : l \pmod p \in R\}$, then for every bipartite graph $G$,

\noindent {\rm(i)}
$\displaystyle \bmrp{G^c}\leq \sum_{t=0}^s{\tl{G}\choose t}.$

\noindent{\rm(ii)}
$\displaystyle \bmrp{G}\leq \sum_{t=0}^{p-1}{\tl{G}\choose t}.$
\end{thm}

\begin{proof}{
Suppose that, ${\cal A}=\{A_1,\dots,A_n\}$ and ${\cal B}=\{B_1,\dots,B_n\}$ are the families of sets assigned to the vertices in two partition sets of $G$ in a set representation by $\theta_L(G)$ labels.
Let $M_t=A_t({\cal A},{\cal B})$ be the $t$-intersection matrix of $\cal A$ and $\cal B$. Now we follow the similar argument as in the proof of Theorem~\ref{mod_mr} and conclude that

\noindent (i) $\bmrp{G^c} \leq \rk{M} \leq \sum_{t=0}^s{\rk{M_t}} \le \sum_{t=0}^s{\tl{G}\choose t},$ where $M=\sum_{t=0}^s a_tM_t$ and $a_t$, $0\le t \le s$,  satisfy in Relation~(\ref{1}).

\noindent (ii) $\bmrp{G} \leq \rk{M} \leq \sum_{t=0}^{p-1}{\rk{M_t}} \le \sum_{t=0}^{p-1}{\tl{G}\choose t},$ where $M= \sum_{i=1}^s\sum_{t=0}^{p-1} b_t^iM_t$ and $b_t^i$, $0 \le t \le p-1$ and $1 \le i \le s$,  satisfy in Relation~(\ref{2}).
}\end{proof}

It is known that if in the above theorem, $L$ is the set of odd numbers, i.e. $p=2$ and $R=\{1\}$, then for every bipartite graph $G$, $\theta_L(G) = \mrf{\Bbb{Z}_2}{G}$~\cite{Jukna_theta}. This shows that the above lower bounds are tight.

From Theorem~\ref{mod_bmr}, by the approximation~(\ref{*}) for the binomial coefficients, we get the following corollary.

\begin{cor}\label{cor_mod_bmr}
Let $p$ be a prime number and $R$ be a subset of residues module $p$ with $|R|=s$, where $s>1$. If $L = \{l : l \pmod p \in R\}$, then for every bipartite graph $G$,\\

\noindent{\rm(i)}
$\tl{G^c} \geq (\bmrp{G})^{1\over s}.$

\noindent{\rm(ii)}
$\tl{G} \geq (\bmrp{G})^{1\over p-1}.$
\end{cor}


By the above lower bounds we obtain an alternative proof of Theorem~\ref{review} as follows.

A bipartite $n\times n$ graph $G=(V_1\cup V_2, E)$ is {\it increasing} if it is possible to enumerate its vertices $V_1=\{x_1, \dots, x_n\}$ and $V_2=\{y_1,\dots, y_n\}$ so that $x_iy_i \in E$ and $x_iy_j\not\in E$ for all $i > j$. In~\cite{Jukna_theta} it is stated that every bipartite $n\times n$ graph $G$ of maximum degree $\Delta$, with no isolated vertices, contains an induced bipartite $(n/\Delta)\times(n/\Delta)$ increasing subgraph.

By Proposition~\ref{known_mr}(\ref{itm:b}), if $H$ is the induced bipartite $(n/\Delta)\times(n/\Delta)$ increasing subgraph of~$G$, then $\bmrf{\Bbb{F}}{G} \geq \bmrf{\Bbb{F}}{H}$. Moreover, the adjacency matrix of $H$ is upper triangular with non-zero diagonal entry. Thus, $\bmrf{\Bbb{F}}{G} \geq n/\Delta$ over any field $\Bbb{F}$. Hence, Corollary~\ref{cor_mod_bmr} implies Theorem~\ref{review}. Furthermore, it should be note that there are graphs that $\bmrf{\Bbb{F}}{G}-n/\Delta = \Omega(n)$.  


\section{Uniform set intersection representation}

In this section, we consider the set intersection representation of graphs which has some constraints on the size of sets assigning to the vertices. In fact, if all sets assigned to the vertices are of the same size, say $k$, then the representation is called the {\it $k$-uniform intersection representation}. The {\it $(L,k)$-intersection number} of $G$, denoted by $\Theta_{L,k}(G)$, is the minimum size of the universe of the sets in all $k$-uniform intersection representations of graph $G$. As a natural extension, we can assume that the size of sets assign to the vertices are restricted to $r$ different sizes in  the set $K=\{k_1,k_2,\dots, k_r\}$. In this case, we denote the minimum size of the universe of the sets in all such representations with $\Theta_{L,K}(G)$. Now we investigate the uniform case and obtain the same lower bounds for $\Theta_{L,k}$ and $\Theta_{L,K}$ for various types $L$.

\begin{thm}\label{mod_uni}
Let $p$ be a prime number and $R$ be a subset of residues module $p$ with $|R|=s$. If $L = \{l : l \pmod p \in R\}$, then for every graph $G$,
$$\mrp{G^c}\leq {\Theta_{L,k}(G)\choose s}.$$
%
\end{thm}

\begin{proof}{
Let ${\cal A}=\{A_1,\dots,A_n\}$ be the $k$-uniform family of subsets assigned to the vertices of $G$ by $\Theta_{L,k}$ labels. Suppose that $M_t=A_t({\cal A},{\cal A})=I({\cal A},t)I({\cal A},t)^T$ be the $t$-intersection matrix of $\cal A$. Remind that, $M_t$ is an $n\times n$ matrix, with ${|A_u\cap A_v|\choose t}$ in position $(u,v)$.

By Proposition~\ref{relat},
$$I({\cal A},s) \times I(s,t) = {k-t \choose s-t} I({\cal A},t).$$

Note that, the column vector space of $M_t$ is a subspace of column vector space of $I({\cal A},t)$.  Moreover, if $0 \leq t \leq s \leq k$, then ${k-t \choose s-t}\neq 0$. Thus, by the above relation, the column vector space of $I({\cal A},t)$ is a subspace of column vector space of $I({\cal A},s)$.

\noindent (i) We define an $n \times n$ matrix $M=\sum_{t=0}^s a_tM_t$ where $a_t$ is in $\Bbb{Z}_p$, $0\leq t\leq s$, satisfying in Relation~(\ref{1}). By the definition of $M$ the column vector space of $M$ is a subspace of the vector space spanned by the columns of $M_t$, $0\leq t\leq s$, that is the subspace of the column vector space of $I({\cal A},s)$, therefore,
$$\rk{M} \leq \rk{I({\cal A},s)}\leq {\Theta_{L,k}(G)\choose s}.$$

On the other hand, by choosing $a_t$, $M$ is a symmetric matrix such that for every $u \neq v$, its $(u,v)$ entry is zero if and only if vertex $u$ is adjacent to vertex $v$. Therefore, over the field $\Bbb{Z}_p$,
$$\mrp{G^c} \leq \rk{M}\leq {\Theta_{L,k}(G)\choose s}.$$

}\end{proof}




A natural extension of the uniform representation is a set representation with the restriction on the size of sets to $r$ different sizes. For such representation, in the next theorem a generalization of the results of Theorem~\ref{mod_uni} is proved.
\begin{thm}\label{ttkl}
If $L=\{l_1,\dots,l_s\}$ and $K=\{k_1,\dots,k_r\}$ are two subsets of non-negative integers, where $k_i > s - r$, $1\le i \le r$, then for every graph $G$, 
$$\mr{G^c}\leq r \sum_{t=s-r+1}^s{\ttkl{G}\choose t}.$$

\end{thm}

\begin{proof}{
Let ${\cal A}=\{A_1,\dots,A_n\}$ be the family of sets assigned to the vertices of $G$ with $\Theta_{L,K}(G)$ labels and ${\cal A}_i$, $1\leq i\leq r$, be the $k_i$-uniform subfamily of subsets of ${\cal A}$. Suppose that $M_t={A}_t({\cal A},{\cal A})$ be the $t$-intersection matrix of $\cal A$.
We define matrix $M=\sum_{t=0}^{s} a_tM_t$, where $a_t$  in $\Bbb{R}$, $0\leq t\leq s$, satisfy in Relation~(\ref{1}).

For convenience, we denote the row and column vector spaces of a matrix $Q$ by $R(Q)$ and $C(Q)$, respectively.
By the definition of $M_t=I({\cal A},t) \times I({\cal A},t)^T$,
\begin{equation*}\label{8}
C(M_t) \subseteq C( I({\cal A},t) ).
\end{equation*}

Moreover, by Proposition~\ref{relat}, we have,
$$I({\cal A}_j,s-r+1) \times I(s-r+1,t) = {k_j-t \choose s-r+1-t} I({\cal A}_j,t).$$

If $0 \leq t \leq s-r+1$ then $t \leq s-r+1 \leq k_j$ and ${k_j-t \choose s-r+1-t}\neq 0$. Hence, by the above equality, for $0 \leq t \leq s-r+1$,
$$C( I({\cal A}_j,t)) \subseteq C( I({\cal A}_j,s-r+1) ).$$

Therefore, we have the following relations, 
$$\begin{array}{lllll}
C(M) & \subseteq & \sum_{t=0}^{s} C(M_t) & \subseteq & \sum_{t=0}^{s} C( I({\cal A},t) ) \\[2mm]
&&& \subseteq & \sum_{t=0}^{s} \sum_{j=1}^{r}C( I({\cal A}_j,t) )\\[2mm]
&&& \subseteq &  \sum_{j=1}^{r} \sum_{t=s-r+1}^{s} C( I({\cal A}_j,t) ).
\end{array}$$

Thus,
$$\begin{array}{lll}
\rk{M} & \leq & \sum_{j=1}^{r}\sum_{t=s-r+1}^{s} |C( I({\cal A}_j,t) )|\\[2mm]
& \leq & \sum_{j=1}^{r} \sum_{t=s-r+1}^{s} {\ttkl{G}\choose t}\\[2mm]
& = & r\sum_{t=s-r+1}^{s} {\ttkl{G}\choose t}.
\end{array}$$

On the other hand, by choosing $a_i$, $M$ is a matrix such that for every $u$ and  $v$, its $(u,v)$ entry is zero if and only if vertex $u$ is adjacent to vertex $v$. Therefore, over field $\Bbb{R}$,
$$\mr{G^c} \leq \rk{M}\leq r\sum_{t=s-r+1}^s{\ttl{G}\choose t}.$$

}\end{proof}

By the same argument as in Theorems~\ref{mod_uni} and~\ref{ttkl}, the same lower bounds for the bipartite version can be obtained.


\end{document}